\documentclass[12pt,reqno]{article}

\usepackage[usenames]{color}
\usepackage{amssymb}
\usepackage{graphicx}
\usepackage{amscd}

\usepackage{graphics,amsmath,amssymb,relsize}
\usepackage{bm}
\usepackage{amsthm}
\usepackage{amsfonts}
\usepackage{latexsym}
\usepackage{epsf}

\newtheorem{theorem}{Theorem}[section]
\newtheorem{corollary}[theorem]{Corollary}
\newtheorem{lemma}[theorem]{Lemma}
\newtheorem{proposition}[theorem]{Proposition}

\newtheorem{defin}[theorem]{Definition}

\newtheorem{examp}[theorem]{Example}
\newenvironment{example}{\begin{examp}\normalfont\quad}{\end{examp}}
\newtheorem{rema}[theorem]{Remark}
\newtheorem{prob}[theorem]{Problem}

\numberwithin{equation}{section}

\newcommand{\bt}{\begin{thm}}
\newcommand{\et}{\end{thm}}
\newcommand{\bp}{\begin{proof}}
\newcommand{\ep}{\end{proof}}
\newcommand{\bprop}{\begin{prop}}
\newcommand{\eprop}{\end{prop}}
\newcommand{\bl}{\begin{lemma}}
\newcommand{\el}{\end{lemma}}
\newcommand{\bc}{\begin{corollary}}
\newcommand{\ec}{\end{corollary}}
\newcommand{\Z}{\mathbb{Z}}
\newcommand{\C}{\mathbb{C}}
\newcommand{\be}{\begin{enumerate}}
\newcommand{\ee}{\end{enumerate}}

\newcommand{\OMIT}[1]{}

\title{Restricted linear congruences}
\author{Khodakhast Bibak \thanks{Department of Computer Science, University of Victoria, Victoria, BC, Canada V8W 3P6. Email: {\tt \{kbibak,bmkapron,srinivas\}@uvic.ca}} \and Bruce M. Kapron \footnotemark[1] \and Venkatesh Srinivasan \footnotemark[1] \thanks{Centre for Quantum Technologies, National University of Singapore, Singapore 117543.} \and Roberto Tauraso \thanks{Dipartimento di Matematica, Universit\`a di Roma ``Tor Vergata'', 00133 Roma, Italy. Email: {\tt tauraso@mat.uniroma2.it}} \and L\'aszl\'o T\'oth
\thanks{Department of Mathematics, University of P\'ecs, 7624 P\'ecs, Hungary. Email: {\tt ltoth@gamma.ttk.pte.hu}}}

\begin{document}

\maketitle

\begin{abstract}
In this paper, using properties of Ramanujan sums and of the discrete Fourier transform of arithmetic functions, we give an explicit formula for the number of solutions of the linear congruence
$a_1x_1+\cdots +a_kx_k\equiv b \pmod{n}$, with $\gcd(x_i,n)=t_i$ ($1\leq i\leq k$), where $a_1,t_1,\ldots,a_k,t_k, b,n$ ($n\geq 1$) are arbitrary integers. As a consequence, we derive necessary and sufficient conditions under which the above restricted linear congruence has no solutions. The number of solutions of this kind of congruence was first considered by Rademacher in 1925 and Brauer in 1926, in the special case of $a_i=t_i=1$ $(1\leq i \leq k)$. Since then, this problem has been studied, in several other special cases, in many papers; in particular, Jacobson and Williams [{\it Duke Math. J.} {\bf 39} (1972), 521--527] gave a nice explicit formula for the number of such solutions when $(a_1,\ldots,a_k)=t_i=1$ $(1\leq i \leq k)$. The problem is very well-motivated and has found intriguing applications in several areas of mathematics, computer science, and physics, and there is promise for more applications/implications in these or other directions.
\end{abstract}


\section{Introduction}\label{Sec 1}

Let $a_1,\ldots,a_k,b,n\in \Z$, $n\geq 1$. A linear congruence in $k$ unknowns $x_1,\ldots,x_k$ is of the form
\begin{align} \label{cong form}
a_1x_1+\cdots +a_kx_k\equiv b \pmod{n}.
\end{align}
By a solution of (\ref{cong form}) we mean an ordered $k$-tuple of integers modulo $n$, denoted by 
$\langle x_1,\ldots,x_k\rangle$, that satisfies (\ref{cong form}). Let $(u_1,\ldots,u_m)$ denote the
greatest common divisor (gcd) of $u_1,\ldots,u_m\in \Z$. The following result, proved by D. N. Lehmer \cite{LEH2}, gives the number of solutions of the above linear congruence:

\begin{proposition}\label{Prop: lin cong}
Let $a_1,\ldots,a_k,b,n\in \Z$, $n\geq 1$. The linear congruence $a_1x_1+\cdots +a_kx_k\equiv b \pmod{n}$ has a solution $\langle x_1,\ldots,x_k \rangle \in \Z_{n}^k$ if and only if $\ell \mid b$, where 
$\ell=(a_1, \ldots, a_k, n)$. Furthermore, if this condition is satisfied, then there are $\ell n^{k-1}$ solutions.
\end{proposition}

Interestingly, this classical result of D. N. Lehmer has been recently used (\cite{BKS4}) in introducing GMMH$^*$ which is a generalization of the well-known $\triangle$-universal hash function family, MMH$^*$.

The solutions of the above congruence may be subject to certain conditions, such as $\gcd(x_i,n)=t_i$ 
($1\leq i\leq k$), where $t_1,\ldots,t_k$ are given positive divisors of $n$. The number of solutions of this kind of congruence, we call it {\it restricted linear congruence}, was investigated in special cases by several authors. It was shown by Rademacher \cite{Rad1925} in 1925 and Brauer \cite{Bra1926} in 1926 that the number $N_n(k,b)$ of solutions of the congruence $x_1+\cdots +x_k\equiv b \pmod{n}$ with the restrictions $(x_i,n)=1$ ($1\leq i\leq k$) is
\begin{align} \label{rest_cong_1}
N_n(k,b)= \frac{\varphi(n)^k}{n} \mathlarger{\prod}_{p\, \mid \, n, \,
p\, \mid\, b} \!\left(1-\frac{(-1)^{k-1}}{(p-1)^{k-1}}
\right)\mathlarger{\prod}_{p\, \mid\, n, \, p \, \nmid \, b}
\!\left(1-\frac{(-1)^k}{(p-1)^k}\right),
\end{align}
where $\varphi(n)$ is Euler's totient function and the products are taken over all prime divisors $p$ of $n$. This result was rediscovered later by Dixon \cite{DIX} and Rearick \cite{REA}. The equivalent formula
\begin{align} \label{rest_cong_1_Ram}
N_n(k,b)= \frac1{n} \mathlarger{\sum}_{d\, \mid\, n} c_d(b) \!\left(c_n\left(\frac{n}{d}\right)\right)^k,
\end{align}
involving the Ramanujan sums $c_n(m)$ (see Section \ref{Sec_2_1}) was obtained by Nicol and Vandiver \cite[Th.\ VII]{NV} and reproved by Cohen \cite[Th.\ 6]{COH0}.

The special case of $k=2$ was treated, independently, by Alder \cite{Ald1958}, Deaconescu \cite{Dea2000}, and Sander \cite{San2009}. For $k=2$ the function $N_n(2,b)$ coincides with Nagell's totient function (\cite{NAG}) defined to be the number of integers $x$ (mod $n$) such that $(x, n)=(b-x, n)=1$. From
(\ref{rest_cong_1}) one easily gets
\begin{align} \label{rest_cong_1 k=2}
N_n(2,b)= n \mathlarger{\prod}_{p\, \mid\, n, \,  p\, \mid\, b}
\!\left(1-\frac{1}{p}\right)\mathlarger{\prod}_{p\, \mid\, n, \, p\, \nmid
\, b} \!\left(1-\frac{2}{p} \right).
\end{align}
From (\ref{rest_cong_1 k=2}) it is clear that $N_n(2,0)=\varphi(n)$ and
\begin{align} \label{rest_cong_1 k=2, b=1}
N_n(2,1)= n\mathlarger{\prod}_{p\, \mid\, n} \!\left(1-\frac{2}{p}\right).
\end{align}

Interestingly, the function $N_n(2,1)$ was applied by D. N. Lehmer \cite{LEH} in studying certain magic squares. It is also worth mentioning that the case of $k=2$ is related to a long-standing conjecture due to D. H. Lehmer from 1932 (see \cite{Dea2000, Dea2006}), and also has interesting applications to Cayley graphs (see \cite{San2009, SanSan2013}).

The problem in the case of $k$ variables can be interpreted as a `restricted partition problem modulo $n$' (\cite{NV}), or an equation in the ring $\Z_n$, where the solutions are its units (\cite{Dea2000, San2009, SanSan2013}). More generally, it has connections to studying rings generated by their units, in
particular in finding the number of representations of an element of a finite commutative ring, say $R$, as the sum of $k$ units in $R$; see \cite{KM} and the references therein. The results of Ramanathan
\cite[Th.\ 5 and 6]{Ram1944} are similar to \eqref{rest_cong_1} and \eqref{rest_cong_1_Ram}, but in another context. See also McCarthy \cite[Ch.\ 3]{MCC} and Spilker \cite{SPI} for further results with
these and different restrictions on linear congruences.

The general case of the restricted linear congruence
\begin{equation} \label{gen_rest_cong}
a_1x_1+\cdots +a_kx_k\equiv b \pmod{n}, \quad (x_i,n)=t_i \ (1\leq i\leq k),
\end{equation}
was considered by Sburlati \cite{Sbu2003}. A formula for the number of solutions of \eqref{gen_rest_cong} was deduced in \cite[Eq.\ (4), (5)]{Sbu2003} with some assumptions on the prime factors of $n$ with
respect to the values $a_i,t_i$ ($1\leq i\leq k$) and with an incomplete proof. The special cases of $k=2$ with $t_1=t_2=1$, and $a_i=1$ ($1\leq i\leq k$) of (\ref{gen_rest_cong}) were considered, respectively, by Sander and Sander \cite{SanSan2013}, and Sun and Yang \cite{SY2014}. Cohen \cite[Th.\ 4, 5]{COH2} derived two explicit formulas for the number of solutions of \eqref{gen_rest_cong} with $t_i=1$, $a_i \mid n$, $a_i$ prime ($1\leq i\leq k$). Jacobson and Williams \cite{JAWILL} gave a nice explicit formula for the number of such solutions when $(a_1,\ldots,a_k)=t_i=1$ $(1\leq i \leq k)$. Also, the special case of $b=0$, $a_i=1$, $t_i=\frac{n}{m_i}$, $m_i\mid n$ ($1\leq i\leq k$) is related to the {\it orbicyclic} (multivariate arithmetic) function (\cite{LIS}), which has very interesting combinatorial and topological applications, in particular in counting non-isomorphic maps on orientable surfaces (see \cite{BKS2, LIS, MENE, MENE2, TOT, WAL}). The problem is also related to Harvey's famous theorem on the cyclic groups of automorphisms of compact Riemann surfaces; see Remark~\ref{sepi app}.

The above general case of the restricted linear congruence (\ref{gen_rest_cong}) can be considered as relevant to the generalized knapsack problem (see Remark~\ref{knapsack}). The {\it knapsack problem} is of
significant interest in cryptography, computational complexity, and several other areas. Micciancio \cite{MIC} proposed a generalization of this problem to arbitrary rings, and studied its average-case
complexity. This {\it generalized knapsack problem}, proposed by Micciancio \cite{MIC}, is described as follows: for any ring $R$ and subset $S \subset R$, given elements $a_1, \ldots , a_k \in R$ and a target element $b \in R$, find $\langle x_1,\ldots,x_k\rangle \in S^k$ such that $\sum_{i=1}^k a_i \cdot x_i = b$, where all operations are performed in the ring.

In the one variable case, Alomair et al. \cite{ACP}, motivated by applications in designing an authenticated encryption scheme, gave a necessary and sufficient condition (with a long proof) for the
congruence $ax\equiv b \pmod{n}$, with the restriction $(x,n)=1$, to have a solution. Later, Gro\v{s}ek and Porubsk\'{y} \cite{GP} gave a short proof for this result, and also obtained a formula for the
number of such solutions. In Theorem~\ref{thm:one var} (see Section~\ref{Sec_3}) we deal with this problem in a more general form as a building block for the case of $k$ variables ($k\geq 1$).

In Section~\ref{Sec_3}, we obtain an explicit formula for the number of solutions of the restricted linear congruence \eqref{gen_rest_cong} for arbitrary integers $a_1,t_1,\ldots,a_k,t_k, b,n$ ($n\geq 1$). Two major ingredients in our proofs are Ramanujan sums and the discrete Fourier transform (DFT) of arithmetic functions, of which properties are reviewed in Section~\ref{Sec_2}. Bibak et al. \cite{BKSTT2} applied this explicit formula in constructing an almost-universal hash function family and gave some applications to authentication and secrecy codes.

\section{Preliminaries}\label{Sec_2}

Throughout the paper we use $(a_1,\ldots,a_k)$ to denote the greatest common divisor (gcd) of $a_1,\ldots,a_k\in \Z$, and write $\langle a_1,\ldots,a_k\rangle$ for an ordered $k$-tuple of integers. Also, for $a \in \Z \setminus \lbrace 0 \rbrace$ and a prime $p$ we use the notation $p^r\, \|\, a$ if $p^r\mid a$ and $p^{r+1}\nmid a$.

\subsection{Ramanujan sums} \label{Sec_2_1}

Let $e(x)=\exp(2\pi ix)$ be the complex exponential with period 1, which satisfies for any $m,n\in\Z$ with $n\ge 1$,
\begin{equation} \label{exp_prop}
\mathlarger{\sum}_{j=1}^n e\!\left(\frac{jm}{n}\right) = \begin{cases}
n, \ & \text{ if $n\mid m$}, \\  0, \ & \text{ if $n\nmid m$}.
\end{cases}
\end{equation}

For integers $m$ and $n$ with $n \geq 1$ the quantity 
\begin{align}\label{def1}
c_n(m) = \mathlarger{\sum}_{\substack{j=1 \\ (j,n)=1}}^{n}
e\!\left(\frac{jm}{n}\right)
\end{align}
is called a {\it Ramanujan sum}. It is the sum of the $m$-th powers of the primitive $n$-th roots of unity, and is also denoted by $c(m,n)$ in the literature.

Even though the Ramanujan sum $c_n(m)$ is defined as a sum of some complex numbers, it is integer-valued (see Theorem~\ref{thm:Ram Mob} below). From (\ref{def1}) it is clear that $c_n(-m) = c_n(m)$. Clearly, $c_n(0)=\varphi (n)$, where $\varphi (n)$ is {\it Euler's totient function}. Also, by Theorem~\ref{thm:Ram Mob} or Theorem~\ref{thm:von rama} (see below), $c_n(1)=\mu (n)$, where $\mu (n)$ is the {\it M\"{o}bius function} defined by
\begin{align}\label{def2}
 \mu (n)&=
  \begin{cases}
    1, & \text{if $n=1$,}\\
    0, & \text{if $n$ is not square-free,}\\
    (-1)^{\kappa}, & \text{if $n$ is the product of $\kappa$ distinct primes}.
  \end{cases}
\end{align}

The following theorem, attributed to Kluyver~\cite{KLU}, gives an explicit formula for $c_n(m)$:

\begin{theorem} \label{thm:Ram Mob}
For integers $m$ and $n$, with $n \geq 1$,
\begin{align}\label{for:Ram Mob}
c_n(m) = \mathlarger{\sum}_{d\, \mid\, (m,n)} \mu
\!\left(\frac{n}{d}\right)d.
\end{align}
\end{theorem}

Thus, $c_n(m)$ can be easily computed provided $n$ can be factored efficiently. By applying the M\"{o}bius inversion formula, Theorem~\ref{thm:Ram
Mob} yields the following property: For $m,n\geq 1$,
\begin{align} \label{Orth1 for cons}
\sum_{d\, \mid\, n} c_{d}(m)&=
  \begin{cases}
    n, & \text{if $n\mid m$},\\
    0, & \text{if $n\nmid m$}.
  \end{cases}
\end{align}

The case $m=1$ of \eqref{Orth1 for cons} gives the {\it characteristic property} of the M\"{o}bius function:
\begin{align}\label{Mob_char}
\sum_{d\, \mid\, n} \, \mu (d)&=
  \begin{cases}
    1, & \text{if $n=1$},\\
    0, & \text{if $n>1$}.
  \end{cases}
\end{align}

Note that Theorem~\ref{thm:Ram Mob} has several other important consequences:

\begin{corollary} \label{cor:Ram Mob} Ramanujan sums enjoy the following properties:

(i) For fixed $m\in \Z$ the function $n\mapsto c_n(m)$ is multiplicative, that is, if $(n_1,n_2)=1$, then
$c_{n_1n_2}(m)=c_{n_1}(m)c_{n_2}(m)$. {\rm(}Note that the function $m\mapsto c_n(m)$ is multiplicative for a fixed $n$ if and only if $\mu(n)=1$.{\rm)} Furthermore, for every prime power $p^r$ {\rm ($r\geq 1$)},
\begin{equation} \label{Ramanujan_prime_power}
c_{p^r}(m)=  \begin{cases}
    p^r-p^{r-1}, & \text{if $p^r\mid m$},\\
    -p^{r-1}, & \text{if $p^{r-1}\, \|\, m$},\\
    0, & \text{if $p^{r-1}\nmid m$}.
  \end{cases}
\end{equation}

(ii) $c_n(m)$ is integer-valued.

(iii) $c_n(m)$ is an {\it even} function of $m \pmod{n}$, that is, $c_n(m)=c_n\left((m,n)\right)$, for every $m,n$.
\end{corollary}

The {\it von Sterneck number} (\cite{von}) is defined by 

\begin{align}\label{def3}
\Phi(m,n)=\frac{\varphi (n)}{\varphi \left(\frac{n}{\left(m,n\right)}\right)}\mu \!\left(\frac{n}{\left(m,n\right)} \right).
\end{align}

A crucial fact in studying Ramanujan sums and their applications is that they coincide with the von Sterneck number. This result is known as {\it von Sterneck's formula} and is attributed to 
Kluyver~\cite{KLU}:

\begin{theorem} \label{thm:von rama}
For integers $m$ and $n$, with $n \geq 1$, we have
\begin{align}\label{von rama for}
\Phi(m,n)=c_n(m).
\end{align}
\end{theorem}

Ramanujan sums satisfy several important {\it orthogonality properties}. One of them is the following identity:
\begin{theorem} {\rm (\cite{COH1})} \label{thm:Orth1}
If $n \geq 1$, $d_1 \mid n$, and $d_2 \mid n$, then we have
\begin{align}\label{Orth1 for}
\mathlarger{\sum}_{d\, \mid\, n} c_{d_1}\! \left(\frac{n}{d}\right)c_{d}\! \left(\frac{n}{d_2}\right)&=
  \begin{cases}
    n, & \text{if $d_1=d_2$},\\
    0, & \text{if $d_1\not=d_2$}.
  \end{cases}
\end{align}
\end{theorem}

We close this subsection by mentioning that, very recently, Fowler et al. \cite{FGK} showed that many properties of Ramanujan sums can be deduced (with very short proofs!) using the theory of {\it supercharacters} (from group theory), recently developed by Diaconis-Isaacs and Andr\'{e}.

\subsection{The discrete Fourier transform}

A function $f:\Z \to \C$ is called {\it periodic} with period $n$ (also called {\it $n$-periodic} or {\it periodic} modulo $n$) if $f(m + n) = f(m)$, for every $m\in \mathbb{Z}$. In this case $f$ is determined by the finite vector $(f(1),\ldots,f(n))$. From (\ref{def1}) it is clear that $c_n(m)$ is a periodic function of $m$ with period $n$.

We define the {\it discrete Fourier transform} (DFT) of an $n$-periodic function $f$ as the function 
$\widehat{f}={\cal F}(f)$, given by
\begin{align}\label{FFT1}
\widehat{f}(b)=\mathlarger{\sum}_{j=1}^{n}f(j)e\! \left(\frac{-bj}{n}\right)\quad (b\in \Z).
\end{align}

The standard representation of $f$ is obtained from the Fourier representation $\widehat{f}$ by
\begin{align}\label{FFT2}
f(b)=\frac1{n} \mathlarger{\sum}_{j=1}^{n}\widehat{f}(j)e\!\left(\frac{bj}{n}\right) \quad (b\in \Z),
\end{align}
which is the {\it inverse discrete Fourier transform} (IDFT); see, e.g., \cite[p.\ 109]{MOVA}.

The {\it Cauchy convolution} of the $n$-periodic functions $f_1$ and $f_2$ is the $n$-periodic function $f_1\otimes f_2$ defined by
\begin{equation*}
(f_1\otimes f_2)(m)= \sum_{\substack{1\leq x_1,x_2\leq n \\ x_1+x_2\equiv m \text{ {\rm (mod $n$)}}}} f_1(x_1)f_2(x_2) =
\sum_{x=1}^n f_1(x)f_2(m-x)  \quad (m\in \Z).
\end{equation*}

It is well known that $$\widehat{f_1\otimes f_2} = \widehat{f_1}\widehat{f_2},$$ with pointwise multiplication. More generally, if
$f_1,\ldots,f_k$ are $n$-periodic functions, then
\begin{equation} \label{Cauchy_prop}
{\cal F}(f_1\otimes \cdots \otimes f_k) = {\cal F}(f_1) \cdots {\cal F}(f_k).
\end{equation}

For $t\mid n$, let $\varrho_{n,t}$ be the $n$-periodic function defined for every $m\in \Z$ by
\begin{equation*}
\varrho_{n,t}(m)=   \begin{cases}
    1, & \text{if $(m,n)=t$},\\
    0, & \text{if $(m,n)\neq t$}.
  \end{cases}
\end{equation*}

We will need the next two results. The first one is a direct consequence of the definitions.

\begin{theorem} \label{Th_varrho}  For every $t\mid n$,
\begin{equation*}
\widehat{\varrho_{n,t}}(m)= c_{\frac{n}{t}}(m) \quad (m\in \Z),
\end{equation*}
in particular, the Ramanujan sum $m\mapsto c_n(m)$ is the DFT of the function $m\mapsto \varrho_{n,1}(m)$.
\end{theorem}

As already mentioned in Corollary~\ref{cor:Ram Mob}(iii), a function $f:\Z \to \C$ is called $n$-even, or even (mod $n$), if
$f(m)=f((m,n))$, for every $m\in \Z$. Clearly, if a function $f$ is $n$-even, then it is $n$-periodic.
The Ramanujan sum $m\mapsto c_n(m)$ is an example of an $n$-even function.

\begin{theorem} {\rm (\cite[Prop.\ 2]{TOTHAU})} \label{Th_TOTHAU}  If $f$ is an $n$-even function, then
\begin{equation*}
\widehat{f}(m)= \sum_{d\, \mid\, n} f(d) c_{\frac{n}{d}}(m) \quad (m\in \Z).
\end{equation*}
\end{theorem}

\begin{proof} Group the terms of \eqref{FFT1} according to the values $d=(m,n)$, taking into account the definition of the $n$-even functions.
\end{proof}

\section{Linear congruences with $(x_i,n)=t_i$ ($1\leq i\leq k$)}\label{Sec_3}

In this section, using properties of Ramanujan sums and of the discrete Fourier transform of arithmetic functions, we derive an explicit formula for the number of solutions of the restricted linear congruence \eqref{gen_rest_cong} for arbitrary integers $a_1,t_1,\ldots,a_k,t_k, b,n$ ($n\geq 1$).

Let us start with the case that we have only one variable; this is a building block for the case of $k$ variables ($k\geq 1$). The following theorem generalizes the main result of \cite{GP}, one of the main results of \cite{ACP}, and also a key lemma in \cite{NV} (Lemma 1).

\begin{theorem} \label{thm:one var}
Let $a$, $b$, $n\geq 1$ and $t\geq 1$ be given integers. The congruence $ax \equiv b \pmod{n}$ has solution(s) $x$ with $(x,n)=t$ if and only
if $t\mid(b,n)$ and $\left(a,\frac{n}{t}\right)=\left(\frac{b}{t},\frac{n}{t}\right)$. Furthermore, if these conditions are satisfied,
then there are exactly
\begin{equation} \label{number_one_var}
\frac{\varphi\left(\frac{n}{t}\right)}{\varphi\left(\frac{n}{td}\right)}= d \prod_{\substack{p\, \mid\, d\\ p\, \nmid\, \frac{n}{td}}}
\left(1-\frac1{p} \right)
\end{equation}
solutions, where $p$ ranges over the primes and $d = \left(a,\frac{n}{t}\right)=\left(\frac{b}{t},\frac{n}{t}\right)$.
\end{theorem}

\begin{proof}
Assume that there is a solution $x$ satisfying $ax \equiv b \pmod{n}$ and $(x,n)=t$. Then $(ax,n)=(b,n)=td$, for some $d$.
Thus, $t\mid (b,n)$ and $\left(\frac{ax}{t},\frac{n}{t}\right)=\left(\frac{b}{t},\frac{n}{t}\right)=d$. But
since $\left(\frac{x}{t},\frac{n}{t}\right)=1$, we have $\left(a,\frac{n}{t}\right)=\left(\frac{b}{t},\frac{n}{t}\right)=d$.

Now, let $t\mid (b,n)$ and $\left(a,\frac{n}{t}\right)=\left(\frac{b}{t},\frac{n}{t}\right)=d$. Let us denote
$A=\frac{a}{d}$, $B=\frac{b}{dt}$, $N=\frac{n}{dt}$. Then $(A,N)=(B,N)=1$. Since $(A,N)=1$, the congruence $Ay \equiv B \pmod{N}$ has
a unique solution $y_0=A^{-1}B$ modulo $N$ and $(Ay_0,N)=(B,N)$, that is $(y_0,N)=1$. It follows that $a(ty_0) \equiv b \pmod{n}$, which shows that $x_0=ty_0$ is a solution of $ax \equiv b \pmod{n}$.

If $x$ is such that $ax \equiv b \pmod{n}$ and $(x,n)=t$, then $x=ty$ and $Ay \equiv B \pmod{N}$. Hence, all solutions of the congruence $ax \equiv b \pmod{n}$ with $(x,n)=t$ have the form $x=t(y_0+kN)$, where $0\leq k\leq d-1$ and $\left(y_0+kN,\frac{n}{t}\right)=1$. Since $(y_0,N)=1$, the latter condition is equivalent to $(y_0+kN,d)=1$. The number $S$ of such solutions, using the characteristic property \eqref{Mob_char} of the M\"{o}bius function, is
\begin{align*}
S&= \sum_{\substack{0\leq k\leq d-1 \\ (y_0+kN,d)=1}} 1\\
&= \sum_{0\leq k\leq d-1} \sum_{\delta \, \mid\, (y_0+kN,d)} \mu(\delta)\\
&= \sum_{\delta \, \mid\, d} \mu(\delta) \sum_{\substack{0\leq k\leq d-1\\ \delta \, \mid \, y_0+kN}} 1 = \sum_{\delta \, \mid\, d} \mu(\delta) \sum_{\substack{0\leq k\leq d-1\\ kN\equiv -y_0 \pmod{\delta}}} 1.
\end{align*}
Here, if $v=(N,\delta)>1$, then $v\nmid y_0$ since $(y_0,N)=1$. Thus, the congruence $kN\equiv -y_0 \pmod{\delta}$ has no solution in $k$ and the inner sum is zero. If $(N,\delta)=1$, then the same congruence has one solution in $k$ (mod $\delta$) and it has $\frac{d}{\delta}$ solutions
(mod $d$). Therefore,
$$
S=\sum_{\substack{\delta \, \mid\, d\\ (\delta,N)=1}} \mu(\delta)\frac{d}{\delta} = d\prod_{\substack{p\, \mid\, d\\p\, \nmid\, N}} \left(1-\frac1{p}\right) =\frac{\varphi(Nd)}{\varphi(N)}=
\frac{\varphi\left(\frac{n}{t}\right)}{\varphi\left(\frac{n}{td}\right)}.$$
The proof is now complete.
\end{proof}

\begin{rema}
In \cite{ACP} the authors only prove the first part of Theorem~\ref{thm:one var} in the case of $t=1$,  and apply the result in checking the integrity of their authenticated encryption scheme (\cite{ACP}). Their main result, \cite[Th. 5.11]{ACP}, is obtained via a very long argument; however, formula \eqref{number_one_var} alone gives a one-line proof for \cite[Th. 5.11]{ACP} that we omit here.
\end{rema}

\begin{corollary} \label{cor_one_sol} The congruence $ax \equiv b \pmod{n}$ has exactly one solution $x$ with $(x,n)=t$ if and only if one of the following two cases holds:

(i) $\left(a,\frac{n}{t}\right)=\left(\frac{b}{t},\frac{n}{t}\right)=1$, where $t\mid (b,n)$;

(ii) $\left(a,\frac{n}{t}\right)=\left(\frac{b}{t},\frac{n}{t}\right)=2$, where $t\mid b$, $n=2^r u$, $r\ge 1$, $u\ge 1$ odd, $t=2^{r-1}v$,
$v\mid u$.
\end{corollary}

\begin{proof} Let $d=\left(a,\frac{n}{t}\right)=\left(\frac{b}{t},\frac{n}{t}\right)$. If $d=1$, then \eqref{number_one_var} shows that there is one solution. Now for $d>1$ it is enough to consider the case when $d=p^j$ ($j\ge 1$) is a prime power. Let $p^r\, \|\, n$, $p^s\, \|\, t$ with $0\leq j+s\leq r$. Then, by \eqref{number_one_var}, there is one solution if $p^j\left(1-\frac1{p}\right)=1$ provided that $p\nmid p^{r-s-j}$. This holds only in the case $p=2$, $j=1$, $s+j=r$. This gives $d=2$ together with the conditions formulated in (ii).
\end{proof}

\noindent{We remark that Corollary~\ref{cor_one_sol}, in the case of $t=1$, was obtained in \cite[Cor.\ 4]{GP}.}

Now we deal with the case of $k$ variables ($k\geq 1$). Assume $a_1,\ldots,a_k,b$ are fixed and let $N_n(t_1,\ldots,t_k)$ denote the number of incongruent solutions of \eqref{gen_rest_cong}. We note the following multiplicativity property: If $n,m\ge 1$, $(n,m)=1$, then
\begin{equation} \label{multipl_prop}
N_{nm}(t_1,\ldots,t_k) = N_n(u_1,\ldots,u_k) N_m(v_1,\ldots,v_k),
\end{equation}
with unique $u_i, v_i$ such that $t_i=u_iv_i$, $u_i\mid n$, $v_i\mid m$ ($1\leq i\leq k$). This can be easily shown by the Chinese remainder theorem. Therefore, it would be enough to obtain $N_n(t_1,\ldots,t_k)$ in the case $n=p^r$, a prime power. However, we prefer to derive the next compact results, which are valid for an arbitrary positive integer $n$.
\vspace*{2mm}

In the case that $a_i=1$ ($1\leq i\leq k$), we prove the following result:

\begin{theorem} \label{thm:k var1}
Let $b$, $n\geq 1$, $t_i\mid n$ {\rm ($1\leq i\leq k$)} be given integers. The number of solutions of the linear congruence $x_1+\cdots +x_k\equiv b \pmod{n}$, with $(x_i,n)=t_i$ {\rm ($1\leq i\leq k$)}, is
\begin{align}\label{thm:k var1 for}
N_n(b;t_1,\ldots,t_k)= \frac{1}{n}\mathlarger{\sum}_{d\, \mid\, n} c_d(b)\mathlarger{\prod}_{i=1}^{k}c_{\frac{n}{t_i}}\!
\left(\frac{n}{d}\right)\geq 0.
\end{align}
\end{theorem}

\begin{proof} Apply the properties of the DFT. Observe that
\begin{equation*}
(\varrho_{n,t_1}\otimes \cdots \otimes \varrho_{n,t_k})(b)= \mathlarger{\sum}_{\substack{1\leq x_1,\ldots,x_k\leq n \\ x_1+\ldots+x_k\equiv b
\text{ {\rm (mod $n$)}}\\ (x_i,n)=t_i, \ 1\leq i\leq k}} 1
\end{equation*}
is exactly the number $N_n(b;t_1,\ldots,t_k)$ of solutions of the given restricted congruence.

Therefore, by \eqref{Cauchy_prop} and Theorem \ref{Th_varrho},
\begin{equation*}
\widehat{N_n}(b;t_1,\ldots,t_k)= c_{\frac{n}{t_1}}(b) \cdots c_{\frac{n}{t_k}}(b),
\end{equation*}
where the variable for the DFT is $b$ ($n,t_1,\ldots, t_k$ being parameters). Now the IDFT formula \eqref{FFT2} gives
\begin{equation*}
N_n(b;t_1,\ldots,t_k) = \frac1{n} \mathlarger{\sum}_{j=1}^n c_{\frac{n}{t_1}}(j) \cdots c_{\frac{n}{t_k}}(j) e\! \left(\frac{bj}{n} \right).
\end{equation*}

By Corollary~\ref{cor:Ram Mob}(iii) and the associativity of $\gcd$ one has for every $i$ ($1\leq i\leq k$),
\begin{equation} \label{associ gcd}
c_{\frac{n}{t_i}}\left((j,n)\right)
=c_{\frac{n}{t_i}}\left(\left((j,n),\frac{n}{t_i}\right)\right)
=c_{\frac{n}{t_i}}\left(\left(j,\left(n,\frac{n}{t_i}\right)\right)\right)
=c_{\frac{n}{t_i}}\left(\left(j,\frac{n}{t_i}\right)\right)
=c_{\frac{n}{t_i}}\left(j\right).
\end{equation}

The properties \eqref{associ gcd} show that $m\mapsto c_{\frac{n}{t_1}}(m) \cdots c_{\frac{n}{t_k}}(m)$ is an $n$-even function. Now by applying Theorem \ref{Th_TOTHAU} we obtain \eqref{thm:k var1 for}.
\end{proof}

\begin{rema} \label{rem:orbicyclic} Note that a slight modification of the proof of \cite[Prop.\ 21]{TOT} furnishes an alternate proof for Theorem~\ref{thm:k var1}. Sun and Yang \cite{SY2014} obtained a different formula \textnormal{(}with a longer proof\textnormal{)} for the number of solutions of the linear congruence in Theorem~\ref{thm:k var1}, but we need the equivalent formula \eqref{thm:k var1 for} for the purposes of this paper \textnormal{(}see also \cite{BKSTT3} for another equivalent formula\textnormal{)}. We also remark that the special case of $b=0$, $t_i=\frac{n}{m_i}$, $m_i\mid n$ {\rm ($1\leq i\leq k$)} gives the function
$$
E(m_1, \ldots , m_k)=\frac{1}{n}\mathlarger{\sum}_{d\, \mid\, n} \varphi(d)\mathlarger{\prod}_{i=1}^{k}c_{m_i}\! \left(\frac{n}{d}\right),
$$
which was shown in \cite[Prop.\ 9]{TOT} to be equivalent to the {\it orbicyclic} \textnormal{(}multivariate arithmetic\textnormal{)} function defined in \cite{LIS} by
$$
E(m_1, \ldots , m_k):=\frac{1}{n}\mathlarger{\sum}_{q=1}^{n} \mathlarger{\prod}_{i=1}^{k}c_{m_i}(q).
$$
The orbicyclic function, $E(m_1, \ldots , m_k)$, has very interesting combinatorial and topological applications, in particular, in counting non-isomorphic maps on orientable surfaces, and was investigated in \cite{BKS2, LIS, MENE, TOT}. See also \cite{MENE2, WAL}.
\end{rema}

Now, using Theorem~\ref{thm:one var} and Theorem~\ref{thm:k var1}, we obtain the following general formula for the number of solutions of the restricted linear congruence \eqref{gen_rest_cong}.

\begin{theorem} \label{thm:k var1 more} Let $a_i,t_i, b,n\in \Z$, $n\geq 1$, $t_i\mid n$ {\rm ($1\leq i\leq k$)}. The number of solutions of the linear congruence $a_1x_1+\cdots +a_kx_k\equiv b \pmod{n}$, with $(x_i,n)=t_i$ {\rm ($1\leq i\leq k$)}, is
\begin{align} \label{thm:k var1 more for}
N_n(b;a_1,t_1,\ldots,a_k,t_k) & =\frac{1}{n}\left(\mathlarger{\prod}_{i=1}^{k}\frac{\varphi\left(\frac{n}{t_i}\right)}{\varphi\left(\frac{n}{t_id_i}\right)}\right)
\mathlarger{\sum}_{d\, \mid\, n}
c_d(b)\mathlarger{\prod}_{i=1}^{k}c_{\frac{n}{t_id_i}}\! \left(\frac{n}{d}\right)\\
& = \label{thm:k var2 more for}
\frac{1}{n} \left(\mathlarger{\prod}_{i=1}^{k}  \varphi\left(\frac{n}{t_i}\right)\right) \mathlarger{\sum}_{d\, \mid\, n}
c_d(b) \mathlarger{\prod}_{i=1}^{k} \frac{\mu\left(\frac{d}{(a_it_i,d)}\right)}{\varphi\left(\frac{d}{(a_it_i,d)}\right)},
\end{align}
where $d_i=(a_i,\frac{n}{t_i})$ {\rm ($1\leq i\leq k$)}.
\end{theorem}

\begin{proof} Assume that the linear congruence $a_1x_1+\cdots +a_kx_k\equiv b \pmod{n}$ has a solution
$\langle x_1, \ldots, x_k\rangle \in \Z_{n}^k$ with $(x_i,n)=t_i$ ($1\leq i\leq k$). Let $a_ix_i \equiv y_i \pmod{n}$ ($1\leq i\leq k$). Then $(a_ix_i,n)=(y_i,n)=t_id_i$, for some $d_i$ ($1\leq i\leq k$). Thus,
$(\frac{a_ix_i}{t_i},\frac{n}{t_i})=(\frac{y_i}{t_i},\frac{n}{t_i})=d_i$. But since $(\frac{x_i}{t_i},\frac{n}{t_i})=1$, we have $d_i=(a_i,\frac{n}{t_i})=(\frac{y_i}{t_i},\frac{n}{t_i})$.

By Theorem \ref{thm:k var1}, the number of solutions of the linear congruence $y_1+\cdots +y_k\equiv b \pmod{n}$, with $(y_i,n)=t_id_i$ ($1\leq i\leq k$),
is
\begin{align} \label{eq_1}
\frac{1}{n}\mathlarger{\sum}_{d\, \mid\, n} c_d(b)\mathlarger{\prod}_{i=1}^{k}c_{\frac{n}{t_id_i}}\! \left(\frac{n}{d}\right).
\end{align}

Now, given the solutions $\langle y_1, \ldots, y_k\rangle$ of the latter congruence, we need to find the number of solutions of $a_ix_i \equiv y_i \pmod{n}$, with $(x_i,n)=t_i$ ($1\leq i\leq k$). Since 
$(a_i,\frac{n}{t_i})=(\frac{y_i}{t_i},\frac{n}{t_i})=d_i$, by Theorem \ref{thm:one var}, the latter congruence has exactly
\begin{equation} \label{eq_2}
\frac{\varphi(\frac{n}{t_i})}{\varphi(\frac{n}{t_id_i})}
\end{equation}
solutions. Combining \eqref{eq_1} and \eqref{eq_2} we get the formula \eqref{thm:k var1 more for}.

Furthermore, applying von Sterneck's formula, \eqref{von rama for}, we deduce
\begin{equation}\label{eq_3}
c_{\frac{n}{t_id_i}} \left(\frac{n}{d}\right) = \frac{\varphi(\frac{n}{t_id_i})\mu(w_i)}{\varphi(w_i)},
\end{equation}
where, denoting by $[a,b]$ the least common multiple (lcm) of the integers $a$ and $b$,
\begin{equation*}
w_i =
\frac{\frac{n}{t_id_i}}{(\frac{n}{t_id_i},\frac{n}{d})}=
\frac{\frac{n}{t_id_i}}{\frac{n}{[t_id_i,d]}}=
\frac{[t_id_i,d]}{t_id_i} =\frac{d}{(t_id_i,d)}=
\frac{d}{((a_it_i,n),d)}= \frac{d}{(a_it_i,d)}.
\end{equation*}

By inserting \eqref{eq_3} into \eqref{thm:k var1 more for} we get \eqref{thm:k var2 more for}.
\end{proof}

\begin{rema} For fixed $a_i,t_i$ \textnormal{(}$1\leq i\leq k$\textnormal{)} and fixed $n$, the function
\[
b \mapsto N_n(b;a_1,t_1,\ldots,a_k,t_k)
\]
is an even function \textnormal{(}mod $n$\textnormal{)}. This follows from the formula \eqref{thm:k var1 more for}, showing that
\[
N_n(b;a_1,t_1,\ldots,a_k,t_k)
\]
is a linear combination of the functions $b\mapsto c_d(b)$ \textnormal{(}$d\mid n$\textnormal{)}, which are all even \textnormal{(}mod $n$\textnormal{)} by \eqref{for:Ram Mob}. See also \eqref{associ gcd}.
\end{rema}

\begin{rema} In the case of $k=1$, by comparing Theorem \ref{thm:one var} with formula \eqref{thm:k var1 more for} and by denoting $t_1d_1=s$, we obtain, as a byproduct, the following identity, which is similar to \eqref{Orth1 for} \textnormal{(}and can also be proved directly\textnormal{):} If $b,n\in \Z$, $n\geq 1$, and $s\mid n$, then
\begin{align}\label{New orthogonal}
\mathlarger{\sum}_{d\, \mid\, n} c_d(b) c_{\frac{n}{s}}\! \left(\frac{n}{d}\right) &=
  \begin{cases}
    n, & \text{if $(b,n)=s$},\\
    0, & \text{if $(b,n)\neq s$}.
  \end{cases}
\end{align}
\end{rema}

If in \eqref{gen_rest_cong} one has $a_i=0$ for every $1\leq i\leq k$, then clearly there are solutions 
$\langle x_1,\ldots,x_k\rangle$ if and only if $b\equiv 0\pmod{n}$ and $t_i \mid n$ ($1\leq i\leq k$), and in this case there are $\varphi(n/t_1)\cdots \varphi(n/t_k)$ solutions.

Consider the restricted linear congruence \eqref{gen_rest_cong} and assume that there is an $i_0$ such that $a_{i_0}\ne 0$. For every prime divisor
$p$ of $n$ let $r_p$ be the exponent of $p$ in the prime factorization of $n$ and let $m_p$ denote the smallest $j\geq 1$ such that there is some $i$
with $p^j \nmid a_it_i$. There exists a finite $m_p$ for every $p$, since
for a sufficiently large $j$ one has $p^j\nmid a_{i_0}t_{i_0}$. Furthermore, let
$$
e_p = \# \{i: 1\leq i\leq k, p^{m_p}\nmid a_it_i \}.
$$
By definition, $e_p$ is at most the number of $i$ such that $a_i\ne 0$.

\begin{theorem} \label{th_gen_expl}  Let $a_i,t_i, b,n\in \Z$, $n\geq 1$, $t_i\mid n$ {\rm ($1\leq i\leq k$)} and
assume that $a_i\neq 0$ for at least one $i$. Consider the linear
congruence $a_1x_1+\cdots +a_kx_k\equiv b \pmod{n}$, with
$(x_i,n)=t_i$ {\rm ($1\leq i\leq k$)}. If there is a prime $p\mid n$ such that $m_p\leq r_p$ and $p^{m_p-1}\nmid b$ or
$m_p\geq r_p+1$ and $p^{r_p}\nmid b$, then the linear congruence has no solution. Otherwise, the number of solutions is
\begin{equation} \label{main_prod_formula}
\mathlarger{\prod}_{i=1}^{k} \varphi\left(\frac{n}{t_i}\right)
\mathlarger{\prod}_{\substack{p\,\mid\, n \\ m_p \,\leq \, r_p \\ p^{m_p} \,\mid\, b}} p^{m_p-r_p-1} \left(1-\frac{(-1)^{e_p-1}}{(p-1)^{e_p-1}} \right)
\mathlarger{\prod}_{\substack{p\, \mid\, n \\ m_p \,\leq \, r_p \\ p^{m_p-1} \, \|\, b}} p^{m_p-r_p-1} \left(1-\frac{(-1)^{e_p}}{(p-1)^{e_p}}\right),
\end{equation}
where the last two products are over the prime factors $p$ of $n$ with the given additional properties. Note that the last product is empty and equal
to $1$ if $b=0$.
\end{theorem}

\begin{proof} For a prime power $n=p^{r_p}$ ($r_p\ge 1$) the inner sum of \eqref{thm:k var2 more for} is
\begin{equation*}
W:=\mathlarger{\sum}_{d\, \mid\, p^{r_p}} c_d(b) \mathlarger{\prod}_{i=1}^{k}
\frac{\mu\left(\frac{d}{(a_it_i,d)}\right)}{\varphi \left(
\frac{d}{(a_it_i,d)}\right)} = \mathlarger{\sum}_{j=0}^{r_p} c_{p^j}(b)
\mathlarger{\prod}_{i=1}^{k}
\frac{\mu\left(\frac{p^j}{(a_it_i,p^j)}\right)}{\varphi
\left(\frac{p^j}{(a_it_i,p^j)}\right)}.
\end{equation*}

Assume that $m_p\leq r_p$. Then $p^{m_p-1}\mid a_it_i$ for every $i$ and $p^{m_p}\nmid a_it_i$ for at least one $i$. Therefore, $(a_it_i,p^j)=p^j$ if $0\leq j\leq m_p-1$. Also, $(a_it_i,p^{m_p})=p^{m_p-1}$ if $p^{m_p}\nmid a_it_i$, and this holds for $e_p$ distinct values of $i$. We obtain
\begin{equation*}
W= \mathlarger{\sum}_{j=0}^{m_p-1} c_{p^j}(b) + c_{p^{m_p}}(b)
\frac{(-1)^{e_p}}{(p-1)^{e_p}},
\end{equation*}
the other terms are zero. We deduce by using \eqref{Orth1 for cons} and \eqref{Ramanujan_prime_power} that
\begin{align}
W= \begin{cases}
p^{m_p-1}\left(1- \frac{(-1)^{e_p-1}}{(p-1)^{e_p-1}}\right), & \text{if $p^{m_p}\mid b$}, \\
p^{m_p-1}\left(1- \frac{(-1)^{e_p}}{(p-1)^{e_p}}\right), & \text{if $p^{m_p-1}\, \|\, b$}, \\
0, & \text{if $p^{m_p-1}\nmid b$}. \end{cases}
\end{align}

Now assume that $m_p\geq r_p+1$. Then $p^{r_p}\mid a_it_i$ for every $i$ and $(a_it_i,p^j)=p^j$ for every $j$ with $0\leq j\leq r_p$. Hence, by using \eqref{Orth1 for cons},
\begin{equation*}
W= \mathlarger{\sum}_{j=1}^{r_p} c_{p^j}(b)= \begin{cases}
p^{r_p}, & \text{if $p^{r_p}\mid b$}, \\
0, & \text{if $p^{r_p}\nmid b$}. \end{cases}
\end{equation*}

Inserting into \eqref{thm:k var2 more for} and by using the multiplicativity property \eqref{multipl_prop} we deduce that there is no solution in the specified cases. Otherwise, the number of solutions is
given by
\begin{equation*}
\mathlarger{\prod}_{p\, \mid\, n} p^{-r_p}
\mathlarger{\prod}_{i=1}^{k} \varphi\left(\frac{n}{t_i}\right)
\mathlarger{\prod}_{\substack{p\, \mid\, n \\ m_p \, \geq \, r_p+1 \\ p^{r_p}\, \mid\, b}} p^{r_p}
\mathlarger{\prod}_{\substack{p\, \mid\, n \\ m_p \, \leq \, r_p \\ p^{m_p}\, \mid\, b}} p^{m_p-1} \left(1-\frac{(-1)^{e_p-1}}{(p-1)^{e_p-1}} \right)
\end{equation*}
\begin{equation*}
\times \mathlarger{\prod}_{\substack{p\, \mid\, n \\ m_p \, \leq \, r_p \\ p^{m_p-1} \, \|\, b}} p^{m_p-r_p-1} \left(1-\frac{(-1)^{e_p}}{(p-1)^{e_p}}\right),
\end{equation*}
where the multiplicativity property is also applied to the product of the $\varphi$ factors. This gives \eqref{main_prod_formula}.
\end{proof}

\begin{corollary} \label{cor_zero} The restricted congruence given in Theorem~\ref{th_gen_expl} has no solutions if and only if one of the following cases holds:

(i) there is a prime $p\mid n$ with $m_p\leq r_p$ and $p^{m_p-1}\nmid b$;

(ii) there is a prime $p\mid n$ with $m_p\geq r_p+1$ and $p^{r_p}\nmid b$;

(iii) there is a prime $p\mid n$ with $m_p\leq r_p$, $e_p=1$ and $p^{m_p}\mid b$;

(iv) $n$ is even, $m_2 \leq r_2$, $e_2$ is odd and $2^{m_2}\mid b$;

(v) $n$ is even, $m_2 \leq r_2$, $e_2$ is even and $2^{m_2-1}\, \|\, b$.
\end{corollary}

\begin{proof} Use the first part of Theorem~\ref{th_gen_expl} and examine the conditions under which the factors of the products in \eqref{main_prod_formula} vanish.
\end{proof}

\begin{example} \label{examp}

1) Consider $2x_1+x_2+2x_3\equiv 12 \pmod{24}$, with $(x_1,24)=3$, $(x_2,24)=2$, $(x_3,24)=4$.

Here $24=2^3\cdot 3$,

$2\mid a_1t_1=6$, $2\mid a_2t_2=2$, $2\mid a_3t_3=8$,

$2^2\nmid a_1t_1=6$, $2^2\nmid a_2t_2=2$, $2^2\mid a_3t_3=8$, hence $e_2=2$ and $m_2=2$, also $2^2\mid b=12$,

$3\mid a_1t_1=6$, $3\nmid a_2t_2=2$, $3\nmid a_3t_3=8$, hence $e_3=2$, $m_3=1$, also $3^1\mid b=12$.

The number of solutions is
$$
N=\varphi(24/3)\varphi(24/2)\varphi(24/4)
2^{2-3-1}\left(1-\frac{(-1)^{2-1}}{(2-1)^{2-1}}\right)
3^{1-1-1}\left(1-\frac{(-1)^{2-1}}{(3-1)^{2-1}}\right)=8.
$$

2) Now let $2x_1+x_2+2x_3\equiv 4 \pmod{24}$, with $(x_1,24)=3$, $(x_2,24)=2$, $(x_3,24)=4$, where only $b$ is changed.

Here $2^2\mid b=4$, $3^{1-1}\, \|\, b=4$.

The number of solutions is
$$
N=\varphi(24/3)\varphi(24/2)\varphi(24/4)
2^{2-3-1}\left(1-\frac{(-1)^{2-1}}{(2-1)^{2-1}}\right)3^{1-1-1}
\left(1-\frac{(-1)^{2}}{(3-1)^{2}}\right)=4.
$$

3) Let $2x_1+x_2+2x_3\equiv 5 \pmod{24}$, with $(x_1,24)=3$, $(x_2,24)=2$, $(x_3,24)=4$, again only $b$ is changed.

Here $2^{2-1}\nmid b=5$, hence, there are no solutions by Corollary~\ref{cor_zero}(i). (Well, this is obvious, since all terms have to be even, but $5$ is odd.)

4) Let $2x_1+x_2+2x_3\equiv 10 \pmod{24}$, with $(x_1,24)=3$, $(x_2,24)=2$, $(x_3,24)=4$, again only $b$ is changed.

Here $2^{2-1}\, \|\, b=10$, hence, there is no solution by Corollary~\ref{cor_zero}(v).
\end{example}

We believe that Theorem~\ref{th_gen_expl} and Corollary~\ref{cor_zero} are strong tools and may lead to interesting applications/implications. For example, we can connect the restricted linear congruences to the generalized knapsack problem. In fact, Corollary~\ref{cor_zero} helps us to deal with this problem in a quite natural case:

\begin{rema} \label{knapsack}
The generalized knapsack problem with $R=\mathbb{Z}_n$ and $S=\mathbb{Z}_n^{*}$ has no solutions if and only if one of the cases of Corollary~\ref{cor_zero} holds.
\end{rema}

\begin{rema} 
In \cite{BKSTT2}, we applied Theorem~\ref{th_gen_expl} in constructing an almost-universal hash function family using which we gave a generalization of the authentication code with secrecy presented in \cite{ACP}.
\end{rema}

\begin{rema}\label{sepi app}
Very recently, Bibak et al. \cite{BKS2} using Theorem~\ref{th_gen_expl} as the main ingredient proved an explicit and practical formula for the number of surface-kernel epimorphisms from a co-compact Fuchsian group to a cyclic group (see also \cite{MENE}). This problem has important applications in combinatorics, geometry, string theory, and quantum field theory (QFT). As a consequence, they obtained an `equivalent' form of Harvey's famous theorem on the cyclic groups of automorphisms of compact Riemann surfaces (see also \cite{LIS}).
\end{rema}

\begin{rema} If $k=1$ then $e_p=1$ for every prime $p\mid n$, and it is easy to see that from Theorem \ref{th_gen_expl} and Corollary \ref{cor_zero} we reobtain Theorem \ref{thm:one var}.
\end{rema}

The following formula is a special case of Theorem \ref{th_gen_expl} and was obtained by Sburlati \cite{Sbu2003} with an incomplete proof.

\begin{corollary} Assume that for every prime $p\mid n$ one has $m_p=1$, that is $p\nmid a_it_i$ for at least one $i\in \{1,\ldots,k\}$. Then the number of solutions of the restricted linear congruence \eqref{gen_rest_cong} is
\begin{equation} \label{COR}
 \frac{1}{n}\mathlarger{\prod}_{i=1}^{k} \varphi\left(\frac{n}{t_i}\right)
 \mathlarger{\prod}_{p\, \mid\, n,\, p\, \mid\, b} \left(1-\frac{(-1)^{e_p-1}}{(p-1)^{e_p-1}} \right)\mathlarger{\prod}_{p\, \mid\, n, \, p\,
 \nmid\, b} \left(1-\frac{(-1)^{e_p}}{(p-1)^{e_p}}\right).
\end{equation}
\end{corollary}

\section{Concluding remarks}

As we already mentioned, the problem of counting the number of solutions of the linear congruence $a_1x_1+\cdots +a_kx_k\equiv b \pmod{n}$, with $(x_i,n)=t_i$ ($1\leq i\leq k$), is very well-motivated and has found intriguing applications in number theory, combinatorics, geometry, computer science, cryptography, string theory, and quantum field theory. In this paper, we obtained an explicit formula for the number of solutions of this linear congruence in its most general form, that is, for arbitrary integers $a_1,t_1,\ldots,a_k,t_k, b,n$ ($n\geq 1$). As a consequence, we derived necessary and sufficient conditions under which the above restricted linear congruence has no solutions. As this problem has appeared in several areas in mathematics, computer science and physics, we believe that our formulas might lead to more applications/implications in these or other directions.

\section*{Acknowledgements}

During the preparation of this work the first author was supported by a Fellowship from the University of Victoria (UVic Fellowship).

\end{document}